\documentclass[12pt,reqno]{amsart}

\usepackage{amsthm, mathrsfs,amssymb,amsmath}
\usepackage{enumerate}
\usepackage[hidelinks]{hyperref}
\usepackage{xcolor}
\hypersetup{
	colorlinks,
	linkcolor={red!50!black},
	citecolor={green!50!black},
	urlcolor={red!80!black}
}
\makeatletter
\@namedef{subjclassname@2010}{%
	\textup{2010} Mathematics Subject Classification}
\makeatother

\allowdisplaybreaks

\newtheorem{thm}{Theorem}[section]
\newtheorem{lem}[thm]{Lemma}
\newtheorem{prop}[thm]{Proposition}

\theoremstyle{definition}

\theoremstyle{remark}
\newtheorem{rem}{Remark}

\title[Hausdorff-Young inequality for Orlicz spaces]{Hausdorff-Young inequality for Orlicz spaces on compact homogeneous manifolds}
\author{Vishvesh Kumar}
\address{Vishvesh Kumar \endgraf
	Department of Mathematics: Analysis, Logic and Discrete Mathematics
	\endgraf
	Ghent University, Belgium}
\endgraf
\email{vishveshmishra@gmail.com}
\author[M. Ruzhansky]{Michael Ruzhansky}
\address{
	Michael Ruzhansky:
	\endgraf
	Department of Mathematics: Analysis, Logic and Discrete Mathematics
	\endgraf
	Ghent University, Belgium
	\endgraf
	and
	\endgraf
	School of Mathematics
	\endgraf
	Queen Mary University of London
	\endgraf
	United Kingdom
	\endgraf
	{\it E-mail address} {\rm michael.ruzhansky@ugent.be}
}

\begin{document}
	
	\begin{abstract} We  prove the classical Hausdorff-Young inequality  for Orlicz spaces on  compact homogeneous manifolds. 
	\end{abstract}
	\maketitle

	\section{Introduction}  
	\noindent 
	
	The classical Hausdorff-Young inequality is one of the fundamental inequalities in the theory of Fourier analysis on groups.   For a locally compact abelian group $G,$ $1 \leq p \leq 2$ and $p'= \frac{p}{p-1},$ the classical Hausdorff-Young inequality says, `` If $f \in L^1(G) \cap L^p(G)$ then the Fourier transform satisfies  $\widehat{f} \in L^{p'}(\widehat{G})$ and $\|\widehat{f}\|_{p'} \leq \|f\|_p.$ " This inequality was first given by Hausdorff in 1923 for torus $\mathbb{T}.$ Hausdorff was inspired by the work of W. H. Young in 1912 who proved a similar result but did not formulate his result in terms of inequalities. For a historical discussion on this inequality we refer to \cite{SD, HR}. 
	The Hausdorff-Young inequality for a compact group $G$ is given in Hewitt and Ross \cite{HR, HR74}. Later,  Kunze \cite{Kunze} extended it to unimodular groups.  
	On Lebesgue spaces, the Hausdorff-Young inequality is proved by using the Riesz convexity complex interpolation theorem between $p=1$ and $p=2$. It is well known that between any two Lebesgue spaces there is an Orlicz space which is not a Lebesgue space. M. M. Rao \cite{Raoyoung} studied the Hausdorff-Young inequality for Orlicz spaces on locally compact abelian groups. In fact, the celebrated work of M. M. Rao in the context of Orlicz spaces on locally compact groups (see \cite{Raoyoung, Rao1, rao, Rao2}) motivated  the first author to study the Hausdorff-Young inequality for Orlicz spaces on compact hypergroups  \cite{KR,Kumar, Kumar1}. Recently, the second author with his collaborators established non-commutative version of the aforementioned inequality and of Hardy-Littlewood inequalities on compact homogeneous manifolds and locally compact groups \cite{ANR14,AR16,RR,AR}. In this article, we study the classical Hausdorff-Young inequality with an enlargement of the space, namely, an Orlicz space on  compact homogeneous manifolds.   It is to be noted that the Riesz convexity theorem is useful for the $L^p$-spaces only. For Orlicz spaces results are obtained first by extending a key inequality of Hausdorff-Young in the form of Hardy-Littlewood \cite[p. 170]{Hardy}. It is worth mentioning here that we apply the method of Hausdorff-Hardy-Littlewood \cite{Hardy}.
	
	In Section 2, we present the basics of  compact homogeneous manifolds and of Orlicz spaces in the form we use in the sequel. In Section 3, we prove a key lemma which occupies a major part of this section, and finally, we prove the Hausdorff-Young inequality for Orlicz spaces on compact homogeneous manifolds.

	\section{Preliminaries}
	
	\subsection{Fourier analysis on  compact homogeneous manifolds} 
	Let $G$ be a  compact Lie group and let $K$ be a closed subgroup of $G.$  The left coset space $G/K$ can be seen as a homogeneous manifold with respect to the action of $G$ on $G/K$ given by the left multiplication. The homogeneous manifold $G/K$ has a unique normalized $G$-invariant positive Radon measure $\mu$ such that the Weyl formula holds. There exists a unique differential structure for the quotient $G/K.$  Examples of compact homogeneous manifolds are spheres $\mathbb{S}^n\cong \textnormal{SO}(n+1)/\textnormal{SO}(n),$ real projective spaces $\mathbb{RP}^n\cong \textnormal{SO}(n+1)/\textnormal{O}(n),$ complex projective spaces $\mathbb{CP}^n\cong\textnormal{SU}(n+1)/\textnormal{SU}(1)\times\textnormal{SU}(n)$ and more generally Grassmannians $\textnormal{Gr}(r,n)\cong\textnormal{O}(n)/\textnormal{O}(n-r)\times \textnormal{O}(r).$
	
	Let us denote by $\widehat{G}_0$ the subset of $\widehat{G},$  of representations in $G$, that are of class I with respect to the subgroup $K$. This means that $\pi\in \widehat{G}_0$ if there exists at least one non-zero invariant vector $a$ in the representation space $\mathcal{H}_\pi$ with respect to $K,$ i.e., $\pi(h)a=a$ for every $h\in K.$ Let us denote by $B_{\pi}$ the vector space of these invariant vectors and let $k_{\pi}=\dim B_{\pi}.$ We fix the an orthonormal basis of $\mathcal{H}_\pi$ so that the  first $k_\pi$ vectors are the basis of $B_\pi.$ We note that if $K=\{e\},$ then $G/K$ is equal to the Lie group $G$ and in this case $k_\pi=d_\pi$ for all $\pi \in \widehat{G}.$ On the other hand, if $K$ is a massive subgroup then $k_\pi=1.$ This is the case for the sphere $\mathbb{S}^{n}.$ Other examples can be found in \cite{Vilenkin}. 
	
	For a function $f \in C^\infty(G/K)$ we can write the Fourier series of its canonical lifting $\tilde{f}:=f(gK)$ to $G,$ $\tilde{f} \in C^\infty(G),$ so that the Fourier coefficients satisfy $\widehat{\tilde{f}}=0$ for all representations $\pi \notin \widehat{G}_0.$ Moreover, for class $I$ representations we have $\widehat{\tilde{f}}(\pi)_{ij}=0$ for $i>k_\pi.$ We will often drop the tilde for the simplicity and agree that for a distribution $f \in \mathcal{D}'(G/K) $ we have $\widehat{f}(\pi)=0$ for $\pi \notin \widehat{G}_0$ and $\widehat{f}(\pi)_{ij}=0 $ for $i > k_\pi.$  In order to shorten the notation, for $\pi \in \widehat{G}_0$ it make sense to set $\pi(x)_{ij}=0$ for $j >k_{\pi}.$ We can write the Fourier series of $f$ (or of $\tilde{f}$) in terms of the spherical functions $\pi_{ij}, 1 \leq j \leq k_\pi$, of the representation $\pi \in \widehat{G}_0,$ with respect to the subgroup $K.$ The Fourier series of $f \in C^\infty(G/K)$ is given by 
	$$f(x)= \sum_{\pi \in \widehat{G}_0} d_\pi \sum_{i=1}^{d_\pi} \sum_{j=1}^{k_\pi} \widehat{f}(\pi)_{ji} \pi_{ij}(x),$$ which can also be written in a compact form 
	$$f(x)= \sum_{\pi \in \widehat{G}_0} d_{\pi} \text{Tr}(\widehat{f}(\pi) \pi(x)).$$
	
	For the future reference we note that with these conventions the matrix $\pi(x)\pi(x)^*$ is diagonal matrix with first $k_\pi$ diagonal entries equal to one and others are equal to zero. Therefore, we have $\sum_{i=1}^{d_\pi} \sum_{j=1}^{k_\pi} |\pi(x)_{ij}|^2= \text{Tr}(\pi(x)\pi(x)^*)= k_\pi ^{\frac{1}{2}}.$ The $\ell^p$-spaces on $\widehat{G}_0$ can be defined similar to the spaces $\ell^p(\widehat{G})$ defined in \cite{RuzT} (see also \cite{Hirschman}). First, for the  space of Fourier coefficients of functions on $G/K$ we set 
	\begin{equation}
	\Sigma(G/K)=\{\sigma: \pi \mapsto \sigma(\pi) \in \mathbb{C}^{d_\pi \times d_\pi}: \, [\pi] \in \widehat{G}_0,\,\, \sigma(\pi)_{ij}=0\,\text{for}\, i >k_{\pi}\}.
	\end{equation} 
	Now, we define the Lebesgue spaces $\ell^p(\widehat{G}_0) \subset \Sigma(G/K)$ by the condition
	\begin{equation}
	\|\sigma\|_{\ell^p(\widehat{G}_0)}:= \left( \sum_{[\pi] \in \widehat{G}_0} d_\pi k_\pi^{p(\frac{1}{p}-\frac{1}{2})} \|\sigma(\pi)\|_{HS}^p \right)^{\frac{1}{p}},\,\,\,\, 1 \leq p <\infty,
	\end{equation}
	and $$\|\sigma\|_{\ell^\infty(\widehat{G}_0)}:= \sup_{[\pi] \in \widehat{G}_0} k_\pi^{-\frac{1}{2}} \|\sigma(\pi)\|_{HS}.$$ The following embedding properties hold for these spaces: 
	$$ \ell^{p_1}(\widehat{G}_0) \subset \ell^{p_2}(\widehat{G}_0)\,\, \text{and}\,\, \|\sigma\|_{\ell^{p_2}(\widehat{G}_0)} \leq \|\sigma\|_{\ell^{p_1}(\widehat{G}_0)},\,\,\, 1 \leq p_1 \leq p_2 \leq \infty. $$

	The Hausdorff-Young inequality for Lebesgue spaces on compact homogeneous manifolds is proved in \cite{NRT16} which is stated in the  following theorem. 
	\begin{thm} \label{YH}
		Let $G/K$ be a compact homogeneous manifold with normalized Haar measure $\mu$ and let $f \in L^p(G/K)$ for $1\leq p \leq 2.$ Suppose $\frac{1}{p}+\frac{1}{q}=1,$ then we have 
		\begin{equation} \label{29}
		\|\hat{f}\|_{\ell^q(\widehat{G}_0)} \leq \|f\|_{L^p(G/K)}. 
		\end{equation}   
	\end{thm}
	
	Here we would like to mention that there is a well-established theory of another family of Lebesgue spaces $\ell^p_{\text{sch}}$ on $\widehat{G}_0$ defined using Schatten $p$-norm $\|\cdot\|_{S^p}$ instead of Hilbert-Schmidt norm $\|\cdot\|_{HS}$ on the space of $(d_\pi \times d_\pi)$-dimensional matrices. Indeed, the space $\ell^p_{\text{sch}}(\widehat{G}_0) \subset \Sigma(G/H)$ is defined by the norm 
	\begin{equation}
	\|\sigma\|_{\ell^p_{\text{sch}}(\widehat{G}_0)}:= \left(\sum_{[\pi] \in \widehat{G}_0} d_\pi  \|\sigma(\pi)\|_{S^p}^p \right)^{\frac{1}{p}}\,\,\, \sigma \in \Sigma(G/K),\,\, 1\leq p<\infty,
	\end{equation}
	and
	$$\|\sigma\|_{\ell^\infty_{\text{sch}}(\widehat{G}_0)}:= \sup_{[\pi] \in \widehat{G}_0} \|\sigma(\pi)\|_{\mathcal{L}(\mathcal{H}_\pi)}\,\,\,\, \sigma \in \Sigma(G/K). $$
	
	The following proposition presents the relation between both  norms on $\ell^p(\widehat{G}_0).$ 
	\begin{prop}
		For $1 \leq p \leq 2,$ we have the following  continuous embeddings as  well as the estimates:  
		$\ell^p(\widehat{G}_0) \hookrightarrow \ell^p_{sch}(\widehat{G}_0)$ and $\|\sigma\|_{\ell^p_{sch}(\widehat{G}_0)} \leq \|\sigma\|_{\ell^p(\widehat{G}_0)}$ \, for all $\sigma \in \Sigma(G/K).$ For $2 \leq p \leq \infty,$ we have 
		$\ell^p_{sch}(\widehat{G}_0) \hookrightarrow \ell^p(\widehat{G}_0)  $ and $ \|\sigma\|_{\ell^p(\widehat{G}_0)} \leq  \|\sigma\|_{\ell^p_{sch}(\widehat{G}_0)} $ \, for all $\sigma \in \Sigma(G/K).$
	\end{prop}
	\begin{proof}
		The proof of proposition can be found in \cite{VR}.
	\end{proof}

	The space $\ell^p(\widehat{G}_0)$ and the Hausdorff-Young inequality for it become useful for convergence of Fourier series and characterization of Gevrey-Roumieu ultradifferentiable functions and Gevrey-Beurling ultradifferentiable functions on compact homogeneous manifolds \cite{AR14}. For more detail on Fourier analysis on compact homogeneous manifolds we refer to \cite{Vilenkin,NRT16,Applebaum,AR14}.
	
	\subsection{Basics of Orlicz spaces} For basics of Orlicz spaces one can refer to excellent monographs  \cite{Zygmund, Rao, rao} and articles \cite{Raoyoung, Rao1, Rao2, Kumar}. However we present a few definitions and results here in the form we need. 
	
	A non-zero convex function $\Phi : [0, \infty) \rightarrow [0, \infty]$ is called a {\it Young function} if  $\Phi(0)=0$ and $\lim_{x \rightarrow \infty} \Phi(x)=\infty.$ For any given Young function $\Phi$  the {\it complimentary function} $\Psi$ of $\Phi$  is given by  
	$$ \Psi(y)= \mbox{sup}\{x|y|- \Phi(x): x \geq 0\} \,\,\,\,   (y \geq 0),$$ which is also a Young function. If $\Psi$ is the complementary function of $\Phi$ then $\Phi$ is the complimentary function of $\Psi;$  the pair $(\Psi, \Phi)$ is called a {\it complementary pair}. In fact, a complementary pair of Young functions satisfies $$ xy \leq \Phi(x)+\Psi(y)\,\,\,(x,y \geq 0).$$  If a complimentary pair of Young functions $(\Phi, \Psi)$ satisfies $\Phi(1)+\Psi(1)=1$ then the pair $(\Phi, \Psi)$ is called a {\it normalized complimentary pair}. 
	
	Let $G/K$ be a compact homogeneous manifold with a left Haar measure $\mu.$  Denote the set of all complex valued $\mu$-measurable functions on $G/K$ by $L^0(G/K).$ Given a Young function $\Phi$, the {\it modular function} $\rho_\Phi : L^0(G/K) \rightarrow \mathbb{R}$ is defined by $\rho_\Phi(f):=\int_{G/K} \Phi(|f|)\, d\mu$ and the {\it Orlicz space} is defined by 
	$$L^{\Phi}(G/K) := \left\lbrace f \in L^0(G/K) : \rho_\Phi(af)< \infty~~ \mbox{for ~~some}\,\, a>0  \right\rbrace.$$ 
	Then the Orlicz space is a Banach space with respect to the norm $N_\Phi(\cdot),$ called as {\it Luxumburg norm},  defined by
	$$N_\Phi(f) := \mbox{inf} \left\lbrace \lambda >0: \int_{G/K} \Phi \left(\frac{|f|}{\lambda} \right) \,d\mu \leq \Phi(1) \right\rbrace.$$ 
	
	One can also define the norm $\|\cdot\|_\Phi,$ called {\it Orlicz norm} on $L^\Phi(G/K)$ by 
	$$ \|f\|_\Phi:= \sup \left\lbrace \int_{G/K}|fv|d \mu: \int_{G/K} \Psi \left(|v| \right) \,d\mu \leq \Phi(1) \right\rbrace.$$
	These two norms are equivalent: $\Phi(1) N_\Phi(\cdot) \leq \|\cdot\|_\Phi \leq 2 N_\Phi(\cdot).$ Also, it is known that $N_\Phi(f) \leq 1$ if and only if $\rho_\Phi(f) \leq 1.$ We recover the Lebesgue spaces $L^p,\, 1 \leq p < \infty,$ by considering the Young function $\Phi(x)=\frac{x^p}{p}.$ The complementary young function $\Psi$ corresponding to $\Phi$ is given by $\Psi(x)= \frac{x^q}{q},$ where $q=\frac{p}{p-1}.$  Other examples of Young functions are $e^x-x-1,\, \cosh x-1$ and  $x^p \ln(x).$
	
	The following H$\ddot{\text{o}}$lder inequality holds for Orlicz spaces: for any complementary pair $(\Phi, \Psi)$ and for $f\in L^\Phi(G/K),\,\,g \in L^\Psi(G/K),$ we have
	$$ \int_{G/K} |fg| \, d\mu \leq \min\{ N_\Phi(f)\|g\|_{\Psi},\,\, \|f\|_{\Phi}N_\Psi(g) \}.$$
	
	If $(\Phi, \Psi)$ is a normalized complimentary  pair of Young functions then the above H$\ddot{\text{o}}$lder inequality becomes \cite[P. 58 and P. 62]{Rao} (see also \cite[P. 27]{rao}): 
	\begin{equation} \label{Hol}
	\left| \int_{G/K} fg \, d\mu \right| \leq  N_\Phi(f) \,N_\Psi(g).
	\end{equation}
	
	Let $C(G/K)$ denote the space of complex valued continuous functions on $G/K.$  It can be easily proved that $L^\Phi(G/K) \subset L^1(G/K)$ (for example see \cite{Kumar}) as $G/K$ is compact. Therefore, the closure of $C(G/K)$ inside $L^\Phi(G/K)$ denoted by   $M^\Phi(G/K)$ is same as $L^\Phi(G/K)$ which is not the case in general. A Young function $\Phi$ satisfies the $\Delta_2$ condition if there exist a constant $C >0$ and $x_0\geq 0$ such that $\Phi(2x) \leq C \Phi(x)$ for all $x\geq x_0$. In this case we write $\Phi \in \Delta_2 $. If $\Phi \in \Delta_2$, then it follows that $(L^\Phi)^*= L^\Psi$. If, in addition, $\Psi \in \Delta_2$, then the Orlicz space $L^\Phi$ is a reflexive Banach space.
	
	To define the space $\ell^\Phi(\widehat{G}_0)$ we follow the similar construction as in $\ell^p.$ The Orlicz space $\ell^\Phi(\widehat{G}_0) \subset \Sigma(G/K)$ is defined by the norm 
	\begin{equation}
	N_\Phi(\sigma)= \inf \left\{ \lambda>0: \sum_{\pi \in \widehat{G}_0} \Phi \left( \frac{k_\pi^{-\frac{1}{2}}\|\sigma(\pi)\|_{HS}}{\lambda}\right)\, k_{\pi}d_{\pi} \leq \Phi(1) \right\}.
	\end{equation}
	The partial order $\prec$ on th set of all Young functions is defined as: $\Phi_1 \prec \Phi_2$ whenever $\Phi_1(ax) \leq b \Phi_2(x)$ for $|x| \geq x_0>0$ and $\Phi_2(cx) \leq d \Phi_1(x)$ for all $|x| \leq x_1,$ where $a,b,c,d,x_0$ and $x_1$ are fixed positive constants independent of $x.$ In particular, for $L^p$-spaces with $p \geq 1$ we can see that $a=b=c=d=1, x_1 \geq 1$ and $x_0 \geq 1.$ With the help of this ordering we can define inclusion relation in Orlicz spaces: if $\Phi_1, \Phi_2$ are continuous  Young functions and $\Phi_1 \prec \Phi_2$ then $L^{\Phi_2}(G/K) \subset L^{\Phi_1}(G/K) $ and $N_{\Phi_1}(\cdot) \leq \alpha\, N_{\Phi_2}(\cdot)$ for some $\alpha>0.$ 
	
	Also, for $\widehat{G}_0,$ the space $L^{\Phi}(\widehat{G}_0)$ becomes $\ell^\Phi(\widehat{G}_0)$ under the identification of a matrix valued function $\sigma \in \ell^\Phi(\widehat{G}_0)$ with a non-negative function $\widetilde{\sigma}(\pi)$ defined on $\widehat{G}_0$ by  $\widetilde{\sigma}(\pi):=k_\pi^{-\frac{1}{2}}\|\sigma(\pi)\|_{HS} $ and in this case $\Phi_1 \prec \Phi_2$ implies that $\ell^{\Phi_1}(\widehat{G}_0) \subset \ell^{\Phi_2}(\widehat{G}_0)$ and $N_{\Phi_2}(\cdot) \leq \beta\, N_{\Phi_1}(\cdot)$ for some $\beta >0.$  The following result is well-known (see \cite[Lemma 1, p. 209]{rao}). 
	\begin{lem} \label{le} Let $(\Phi_i, \Psi_i), \,i=1,2$ be complementary pairs of continuous Young functions and let $\Phi_1 \prec \Phi_2$. Then $\Psi_2 \prec \Psi_1$.
	\end{lem}

	\section{The Main result}
	Throughout this section, we assume that $G/K$ is a compact homogeneous manifold and $ \widehat{G}_0$ be the set of type I irreducible inequivalent continuous representations of $G$ for our convenience.  
	From now onwards, we assume that the pair of complimentary continuous Young functions $(\Phi, \Psi)$ is a normalized pair. Note that continuity of a Young function guaranties the existence of its derivative \cite[Corollary 2]{Rao}. Also, since $\Phi$ is a positive
continuous convex function on $[o, \infty)$, it is increasing.
	
	For $f \in L^\Phi(G/K)$ define $F_f:\widehat{G}_0 \rightarrow \mathbb[0, \infty)$ by  \begin{equation} \label{F_f}
	F_f^2(\pi) = \sum_{i=1}^{d_{\pi}} \sum_{j=1}^{k_{\pi}}  \frac{|\widehat{f}(\pi)_{i,j}|^2}{k_{\pi}}  = \frac{\text{Tr}(\widehat{f}(\pi)^*\widehat{f}(\pi))}{k_{\pi}} \,\,\, \text{for} \,\pi \in \widehat{G}_0.
	\end{equation}  
	Now, define the gauge norm of $F_f$ by
	\begin{equation} \label{dis}
	N_\Phi(F_f):= \inf \left\lbrace \lambda>0 : \sum_{\pi \in \widehat{G}_0} \Phi \left( \frac{F_f(\pi)}{\lambda}\right)\, k_{\pi}d_{\pi} \leq \Phi(1) \right\rbrace. 
	\end{equation}
	We note here that for $f \in L^\Phi(G/K),$ $\widehat{f}: \widehat{G}_0  \rightarrow \bigcup_{\pi \in \widehat{G}_0} \mathbb{C}^{d_\pi \times d\pi}.$ So, the norm $N_\Phi(F_f)$ is same as the norm $N_\Phi(\widehat{f})$ on the non-commutative Orlicz space $\ell^\Phi(\widehat{G}_0)$ as $F_f(\pi)= k_\pi^{-\frac{1}{2}}\|\widehat{f}(\pi)\|_{HS}$. The space $(\ell^\Phi(\widehat{G}_0), N_\Phi(\cdot))$ is a  Banach space. If $\Phi$ is continuous then there exists $\lambda_0:= N_\Phi(F_f)$ such that inequality in \eqref{dis} is an equality with $\lambda = \lambda_0$ on the left, i.e.,  $ \sum_{\pi \in \widehat{G}_0} \Phi \left( \frac{F_f(\pi)}{\lambda_0}\right)\,d_{\pi}\, k_{\pi} = \Phi(1).$

	The proof of our main result depends on the following key lemma which is an extension of an important inequality in the case of $L^p$ due to Hardy and Littlewood   \cite{Hardy}.
	
	\begin{lem} \label{m}   Let $G/K$ be a compact homogeneous manifold with the normalized measure $\mu$ and let $(\Phi, \Psi)$ be a pair of continuous normalized Young functions such that \begin{itemize}
			\item[(i)] $\Phi \prec \Phi_0,$ where $\Phi_0(t)= \frac{1}{2}|t|^2,$ 
			\item[(ii)] $\Psi'(t) \leq c_0\, t^p, \,\forall\, t \geq 0,$ for some $p \geq 1,$ and for some $c_0>0.$
		\end{itemize}  Suppose $\Lambda$ is a finite subset of $ \widehat{G}_0.$ Define $f_{\Lambda}: G/K \rightarrow \mathbb{C}$ by 
		\begin{equation} \label{6}
		f_{\Lambda}(x):= \sum_{\pi \in \Lambda} d_{\pi} \sum_{i=1}^{d_{\pi}} \sum_{j=1}^{k_{\pi}} c_{i,j}^\pi \pi_{i,j}(x),
		\end{equation}
		where $c_{i,j}^\pi \in \mathbb{C}.$ If $F_{\Lambda}=F_f$ is as in \eqref{F_f} with $f=f_{\Lambda},$ then
		\begin{equation} \label{7}
		N_\Psi(F_{\Lambda}) \leq \bar{r}_0\, N_\Phi(f_{\Lambda}),
		\end{equation} where $\bar{r_0}>0$  depends only on $\Phi$ and the ordering $ \prec.$
	\end{lem}
	\begin{proof}
		Let $\Lambda$ be a finite subset of $\widehat{G}_0$. If $f_{\Lambda}$ is as in the statement of the lemma,  $\hat{f}_{\Lambda}(\pi)_{i,j}= c^\pi_{i,j}\, \chi_{\Lambda}(\pi)$, where $\chi_{\Lambda}$ is characteristic function of the subset $\Lambda$ in $\widehat{G}_0.$ For simplicity of expressions, we set $S_\Phi(f_{\Lambda})=N_\Phi(F_{f_{\Lambda}}).$ For a non-zero $f \in L^\Phi(G/K),$ the Fourier coefficients $\hat{f}(\pi)_{i,j}$ of $f$ are denoted  by $\tilde{c}^\pi_{i,j}.$ Let $\tilde{f}_{\Lambda}$ be the function $f_\Lambda$ given by \eqref{6} with $c^\pi_{i,j}=\tilde{c}^\pi_{i,j}$.
		Following an idea of  Hardy and Littlewood \cite{Hardy}, we define, 
		\begin{equation} \label{8}
		M= M_\Phi(\Lambda):= \sup \left\lbrace  \frac{S_\Psi(\tilde{f}_{\Lambda})}{N_\Phi(f)} : f \neq 0 \right\rbrace.
		\end{equation} 
		We prove the lemma in three steps.
		
		\noindent {\bf STEP I.} $M<\infty.$ 
		
		\noindent Since $M$ is described by a ratio of norms, without loss of generality we assume that $S_\Psi(\tilde{f}_{\Lambda})=1$ to find a bound on $M$. It follows using continuity of $\Psi$ and the definition of the gauge norm (with $\lambda_0=1,$ see the discussion in the paragraph below the equation \eqref{dis}) that
		\begin{equation}\label{equality13}
		\sum_{\pi \in \widehat{G}_0} \Psi(F_{\tilde{f}_{\Lambda}}(\pi)) k_{\pi}d_{\pi}= \Psi(1).
		\end{equation}
		Since $k_{\pi} \geq 1,$ $F_{\tilde{f}_{\Lambda}}(\pi)=0$ for $\pi \in \widehat{G}_0 \backslash \Lambda$ and $0<\Psi(1)<1,$ at least one term on the left hand side of \eqref{equality13} is greater than or equal to $\frac{\Psi(1)}{\#(\Lambda)},$ where  $\#(\Lambda)$ is the cardinality of $\Lambda.$ If this term is for $\pi=\pi' \in \Lambda$ then we have
		\begin{equation} \label{11}
		0 < \Psi^{-1} \left[ \frac{\Psi(1)}{\#(\Lambda) \,k_{\pi'} \,d_{\pi'}} \right] \leq F_{\tilde{f}_{\Lambda}}(\pi').
		\end{equation}
		Next, 
		\begin{eqnarray*}
			F_{\tilde{f}_{\Lambda}}(\pi)= \left( \frac{1}{k_{\pi}} \sum_{i=1}^{d_\pi} \sum_{j=1}^{k_\pi} |\tilde{c}_{i,j}^\pi|^2 \right)^{\frac{1}{2}}& \leq& \frac{1}{k_{\pi}^{\frac{1}{2}}} \sum_{i=1}^{d_{\pi}} \sum_{j=1}^{k_{\pi}} |\tilde{c}_{i,j}^\pi|\\ & \leq & \frac{1}{k_{\pi}^{\frac{1}{2}}} \int_{G/K} |f(x)| \sum_{i=1}^{d_{\pi}}\sum_{j=1}^{k_{\pi}} |\pi(x)_{i,j}|\, d\mu(x).
		\end{eqnarray*}
		\noindent Using Cauchy-Schwartz inequality, we get 
		\begin{eqnarray*}
			F_{\tilde{f}_{\Lambda}}(\pi) \leq \frac{1}{k_{\pi}^\frac{1}{2}} \int_{G/K} |f(x)| \left( \sum_{i=1}^{d_{\pi}} \sum_{j=1}^{k_{\pi}} |\pi(x)_{i,j}|^2 \right)^{\frac{1}{2}} d_{\pi}^{\frac{1}{2}} k_{\pi}^{\frac{1}{2}}\, d\mu(x).
		\end{eqnarray*}
		Using $\text{Tr}(\pi(x)\pi(x)^*)= k_\pi$,  we get 
		\begin{eqnarray} \label{11'}
		F_{\tilde{f}_{\Lambda}}(\pi) \leq  \int_{G/K} |f(x)| \, d_{\pi}^{\frac{1}{2}} k_{\pi}^{\frac{1}{2}}\, d\mu(x) = d_{\pi}^{\frac{1}{2}} k_{\pi}^{\frac{1}{2}} \int_{G/K} |f(x)| \, d\mu(x).
		\end{eqnarray} 
		Now, by the H$\ddot{\text{o}}$lder inequality  \eqref{Hol} 
		\begin{equation} \label{12}
		F_{\tilde{f}_{\Lambda}}(\pi) \leq d_{\pi}^{\frac{1}{2}} k_{\pi}^{\frac{1}{2}}\, N_\Phi(f).
		\end{equation}
		
		\noindent Now by combining \eqref{11} and \eqref{12}, we have 
		\begin{equation} \label{13}
		\frac{1}{N_\Phi(f)} \leq \left( \frac{d_{\pi'}^{\frac{1}{2}}k_{\pi'}^{\frac{1}{2}}} {\Psi^{-1}\left( \frac{\Psi(1)}{\#(\Lambda) k_{\pi'} d_{\pi'} } \right)} \right) < \infty.
		\end{equation}  
		Since the right hand side of \eqref{13} is independent of $f,$ we have  $M<\infty.$

		\noindent {\bf STEP II.} $M$  is independent of $\Lambda.$
		
		\noindent For $f_{\Lambda}$  as in  \eqref{6}, define $g$ by 
		\begin{equation} \label{18}
		g(x):= \Psi'\left( \frac{|f_{\Lambda}(x)|}{N_{\Psi}(f_{\Lambda})} \right) \text{sgn}(f_{\Lambda}(x)),
		\end{equation}
		where  $\text{sgn}(z)= z/|z|$ for $z \neq 0$ and $0$ for $z=0.$ Since $\Phi$ is continuous and $\mu(G/K)=1$ it follows from the discussion in \cite[p. 175]{Zygmund} (see also \cite[Chapter VI]{rao})  that $N_\Phi(g)=1,$ and that the H$\ddot{\text{o}}$lder's inequality \eqref{Hol} is an equality for the functions $f_{\Lambda}$ and $g$, that is,  
		\begin{eqnarray*}
			N_{\Psi}(f_{\Lambda})  = N_{\Phi}(g) N_{\Psi}(f_{\Lambda})= \left| \int_{G/K} g(x) f_{\Lambda}(x)\,d\mu(x) \right|. 
		\end{eqnarray*}
		Using the Parseval formula, we have  
		\begin{eqnarray*}
			N_{\Psi}(f_{\Lambda}) &=& \left| \sum_{\pi \in \Lambda} d_\pi \sum_{i=1}^{d_{\pi}} \sum_{j=1}^{k_{\pi}} \hat{g}(\pi)_{i,j} \overline{\hat{f}(\pi)}_{i,j} \right| \,\,\,\,\,\, ( \text{by using } \sum_{l} a_l b_l \leq \left(\sum a_l^2 \right)^\frac{1}{2} \left(\sum b_l^2 \right)^\frac{1}{2} )\\ &\leq & \sum_{\pi \in \Lambda} d_\pi \left( \sum_{i=1}^{d_{\pi}} \sum_{j=1}^{k_{\pi}} |\hat{f}(\pi)_{ij}|^2 \right)^{\frac{1}{2}}  \left( \sum_{i=1}^{d_{\pi}} \sum_{j=1}^{k_{\pi}} |\hat{g}(\pi)_{ij}|^2 \right)^{\frac{1}{2}}   \\ &= &  \sum_{\pi \in \Lambda} k_\pi d_\pi\, F_{f_{\Lambda}}(\pi) \,F_{\tilde{g}_{\Lambda}}(\pi)  \,\,\,\,\, \text{(by H$\ddot{\text{o}}$lder's inequality)}\\
			&\leq & N_\Phi(F_{f_{\Lambda}}) N_\Psi(F_{\tilde{g}_{\Lambda}}) = S_\Phi(f_{\Lambda}) S_\Psi(\tilde{g}_{\Lambda}).
		\end{eqnarray*}
		By STEP I, we know that $S_\Psi(\tilde{g}_{\Lambda}) \leq M N_\Phi(g) =M $ and therefore
		\begin{eqnarray} \label{M}
		\frac{N_{\Psi}(f_{\Lambda})}{S_\Phi(f_{\Lambda})} \leq S_\Psi(\tilde{g}_{\Lambda}) \leq M.
		\end{eqnarray}
		\noindent Note that $M \geq 1.$ In fact, for $f=1,$ we note that $N_\Phi(1)=1$ as the measure $\mu$ is normalized. By continuity of $\Phi$, $$\sum_{\pi \in \widehat{G}_0} \Phi \left( \frac{F_{\tilde{f}_{\Lambda}}(\pi)}{S(\tilde{f}_{\Lambda})} \right) k_{\pi} d_{\pi}  = \Phi(1).$$ 
		
		\noindent  Now, by choosing $\pi' \in \Lambda$ such that \eqref{11} holds,  we have 
		\begin{equation}
		S_\Psi(\tilde{f}_{\Lambda}) \geq \frac{1}{\left[ \Phi^{-1}\left( \frac{\Phi(1)}{k_{\pi'}d_{\pi'}}\right)\right]} = r_0\,  \text{(say)}.
		\end{equation}
		Note that $r_0 \geq 1.$ Indeed, since $\Phi$ is increasing and $0<\Phi(1)<1,$ we have $r_0=\frac{1}{\left[ \Phi^{-1}\left( \frac{\Phi(1)}{k_{\pi'}d_{\pi'}}\right)\right]} \geq \frac{1}{\left[ \Phi^{-1}\left( \frac{1}{k_{\pi'}d_{\pi'}}\right)\right]} \geq 1.$
		Therefore,  $M \geq \frac{S_\Psi(\tilde{f}_{\Lambda})}{N_\Phi(f)} \geq r_0 \geq 1.$ By continuity of norms, there is a function $f_{\Lambda}$ such that  
		$M= \frac{N_\Psi(f_{\Lambda})}{S_\Phi(f_{\Lambda})}.$ 
		Consequently, from \eqref{M} we get $S_\Phi(\tilde{g}_{\Lambda})=M.$ We fix this $f_{\Lambda}$ and set $g$ as in \eqref{18} for the remaining part of this step. 
		
		Let us denote $S_\Phi$ by $S_2$ for the Young function $\Phi(x)= \frac{|x|^2}{2}.$ Using the Bessel inequality, we get 
		\begin{equation} \label{19}
		S_2^2(\tilde{g}_{\Lambda}) = \sum_{\pi \in \Lambda} d_{\pi}^2 \sum_{i=1}^{d_{\pi}}\sum_{j=1}^{k_{\pi}} |\hat{g}(\pi)_{i,j}|^2 \leq \int_{G/K} |g|^2 \,d\mu \leq N_\Psi(g^2),
		\end{equation}
		where the last inequality follows from H$\ddot{\text{o}}$lder's  inequality $(\textnormal{since }N_\Phi(1)=1).$ Set $\Psi_1(t)= \Psi(t^2).$ Then $\Psi_1$ is a Young function satisfying $\Psi \prec \Psi_1.$ Since $g$ is a bounded function, by setting $a^2= N_\Psi(g^2)(<\infty),$ we get
		\begin{equation*} 
		\Psi_1(1)= \Psi(1)= \int_{G/K} \Psi\left( \frac{|g|^2}{a^2} \right)\, d\mu(x) = \int_{G/K} \Psi_1\left(\frac{|g|}{a} \right)\, d\mu(x),
		\end{equation*}
		whence $a= N_{\Psi_1}(g).$ Thus, by \eqref{19}, we have  
		\begin{equation}
		S_2(\tilde{g}_{\Lambda}) \leq N_{\Psi_1}(g).
		\end{equation}
		Now, we find an absolute bound for $M$. If $a=N_{\Psi_1}(g)$ then, by definition, there exists $b_0>0$ such that 
		\begin{eqnarray*}
			1= \int_{G/K} \Psi_1\left( \frac{g}{ab_0} \right)d\mu(x) = \int_{G/K} \Psi_1 \left[ \frac{1}{ab_0} \Psi' \left( \frac{|f_{\Lambda}|}{N_\Psi(f_{\Lambda})} \right)\right]d\mu(x).
		\end{eqnarray*}
		Since  $\Psi'(t) \leq c_0\,t^p$ for some $p \geq 1,$ we get 
		\begin{equation} \label{22}
		1 \leq \int_{G/K} \Psi_1 \left[\frac{c_0}{b_0a} \left( \frac{|f_{\Lambda}|}{N_\Psi(f_{\Lambda})}\right)^p \right]d\mu(x) = \int_{G/K} \Psi_2 \left[b_1 \frac{|f_{\Lambda}|}{N_\Psi(f_{\Lambda})} \right] d\mu(x),
		\end{equation}
		where $\Psi_2(t)= \Psi_1(t^p)$ and $b_1= \left( \frac{c_0}{b_0a} \right)^\frac{1}{p}>0.$ Thus $\Psi_2$ is a Young function satisfying $\Psi \prec \Psi_1 \prec \Psi_2.$ Then \eqref{22} gives the following important inequality: there exists a constant $b_2$ depending only on $\Psi_2$ and independent of $f_{\Lambda},$ such that 
		\begin{equation}
		N_{\Psi_2} \left( \frac{b_1f_{\Lambda}}{N_\Psi(f_{\Lambda})} \right) \geq b_2 >0,
		\end{equation}
		(see \cite[Theorem 2, Chapter III]{Rao}).
		Since $N_{\Psi_2}(\cdot)$ is a norm, by the definition of $b_1$, we get 
		\begin{equation}
		\left[\frac{N_{\Psi_2}(f_{\Lambda})}{N_\Psi(f_{\Lambda})} \right]^p \geq b_2^p \left(\frac{b_0}{c_0} \right) N_{\Psi_1}(g) = b_3 N_{\Psi_1}(g), 
		\end{equation}
		where $b_3= b_2^p \frac{b_0}{c_0}.$ Since we have $\Psi_0 \prec \Psi \prec \Psi_1 \prec \Psi_2,$ by Lemma \ref{le}, $\Phi_2 \prec \Phi_1 \prec \Phi \prec \Phi_0$ so that $\ell^{\Phi_2} \subset \ell^{\Phi_1} \subset \ell^\Phi \subset \ell^{\Phi_0}= \ell^2 \subset \ell^\Psi \subset \ell^{\Psi_1} \subset \ell^{\Psi_2}.$ Now, for some $r_2>0,$ we get the following inequalities: 
		\begin{eqnarray} \label{2u}
		1 \leq M= M_{\Phi} = S_\Psi(\tilde{g}_{\Lambda})& \leq & r_2 S_2(\tilde{g}_{\Lambda}) \leq  \frac{r_2}{b_3} \left[\frac{N_{\Psi_2}(f_{\Lambda})}{N_\Psi(f_{\Lambda})} \right]^p \leq \frac{r_2}{b_3} \left[\frac{N_{\Psi_3}(f_{\Lambda})}{N_\Psi(f_{\Lambda})} \right]^p,
		\end{eqnarray}
		where $\Psi_3(t)= \frac{c_0}{p+1} t^{2p(p+1)},$ so $\Psi_2 \prec \Psi_3,$ (Since $\Psi' (t) \leq c_0 \,t^p$ so $\Psi(t)= \frac{c_0}{p+1}t^{p+1}$ and therefore $\Psi_2(t)=\Psi_1(t^p)=\Psi(t^{2p}) \leq \frac{c_0}{p+1}t^{2p(p+1)} = \Psi_3(t) $). Let $\Phi_3$  be the  complementary function of $\Psi_3.$ Then $S_{\Phi_3}(f_{\Lambda}) \leq b_4 S_\Phi(f_{\Lambda})$ for some $b_4>0$ depending only on $\Phi_3$ and $\Phi$ only. Therefore, by \eqref{2u}, we have
		\begin{eqnarray}\label{25}
		1 \leq M= M_{\Phi}&\leq & \frac{r_2}{b_3} \left[ \frac{M_{\Phi_3}S_{\Phi_3}(f_{\Lambda})}{M_\Phi S_{\Phi}(f_{\Lambda})} \right]^p \leq \frac{r_2}{b_3} \left[ b_4 \frac{M_{\Phi_3}}{M_\Phi} \right]^p.
		\end{eqnarray}
		
		\noindent
		Hence,  \begin{equation} \label{27}
		1 \leq M_\Phi^{p+1} \leq r_3^p \,M_{\Phi_3}^p,
		\end{equation}
		for some positive constant $r_3>0$ which depend only on $\Phi, \Phi_2, \Phi_3$ and the ordering constants. But note that $L^{\Psi_3}(G/K)= L^{p'}(G/K),$ where $p'=2p(p+1) \geq 2.$ It follows from Theorem \ref{YH} that  $M_{\Phi_3} \leq r_4< \infty,$  for a positive constant $r_4$ depending on $c_0$ and $r.$ Therefore \eqref{27} gives $$1 \leq M_\Phi \leq r_5< \infty $$ where $r_5= (r_3r_4)^{\frac{p}{p+1}},$ which is independent of $\Lambda.$
		
		\noindent {\bf STEP III.} By setting $\bar{r}_0= r_5,$ equations \eqref{8} immediately give the required inequality in \eqref{7}. Since $\Psi_2$ and $\Psi_3$  depend on the complimentary Young function $\Psi$ of $\Phi$, all the  constants involve depend on $\Phi$ and the ordering $\Phi \prec \Phi_0,$ and perhaps on $c_0$ and $p.$ This completes the proof of the lemma. 
	\end{proof} 
	
	Now, we are ready to prove the Hausdorff-Young inequality for Orlicz spaces on compact homogeneous manifolds.
	\begin{thm} \label{OYH}  Let $G/K$ be a compact homogeneous manifold with the normalized measure $\mu$ and let $(\Phi, \Psi)$ be a pair of continuous normalized Young functions such that \begin{itemize}
			\item[(i)] $\Phi \prec \Phi_0,$ where $\Phi_0(t)= \frac{1}{2} t^2,$ 
			\item[(ii)] $\Psi'(t) \leq c_0\, t^p, \, t \geq 0,$ for some $p \geq 1,$ 
		\end{itemize} where $c_0$ is a positive constant. If $f \in L^\Phi(G/K)$ then there is  $r_0\ge1$ such that 
		$$N_\Psi(F_f) \leq r_0\, N_\Phi(f).$$
	\end{thm}
	\begin{proof} 
		Let $f \in L^\Phi(G/K)$ and let $\Lambda $ be a finite subset of $\widehat{G}_0.$ Suppose $\tilde{f}_{\Lambda}$ is given by \eqref{6} where ${c}_{i,j}^\pi= \hat{f}(\pi)_{i,j}.$ The set $\lbrace \tilde{f}_{\Lambda}: \Lambda \subset \widehat{G}_0 \rbrace$ is the collection of all simple functions which, in particular, contains the set of all matrix coefficients of type I representation of $G$ and therefore it is dense in $L^\Phi(G/K) \subset L^2(G/K)$. We have $\lim_{\Lambda \subset \widehat{G}_0} N_\Phi(f-\tilde{f}_{\Lambda})=0$ and therefore, $\tilde{c}_{ij}^\pi= \widehat{f}_{\Lambda}(\pi)_{ij} \rightarrow  \widehat{f}(\pi)_{ij}=\tilde{c}_{ij}^\pi.$ Consequently,  we have $$\lim_{\Lambda \subset \widehat{G}_0} N_\Psi(F_{\tilde{f}_{\Lambda}})= N_\Psi(F_f),$$
		where the limit as $\Lambda$ varies in $\widehat{G}_0$ is taken using the partial order defined by inclusion of subsets of $\widehat{G}_0.$
		Now, by using this with the inequality $N_\Psi(F_{f_{\Lambda}}) \leq \bar{r}_0\, N_\Phi(f_{\Lambda})$ of \eqref{7} of Lemma \ref{m}, we get $N_\Psi(F_f) \leq r_0\, N_\Phi(f),$  where $r_0=\bar{r}_0.$  This completes the proof. 
	\end{proof}
	\begin{rem} \begin{itemize}
			\item[(i)] For the Lebesgue space, we have that constant $r_0=1$  (see Theorem \ref{YH}); however it is clear from the proof of Lemma \ref{m} that the constant $r_0$ in Theorem \ref{OYH}  is greater or equal to 1. 
			\item[(ii)] We give an example of a pair of Young functions $(\Phi, \Psi)$ satisfying the condition of Theorem \ref{OYH} such that the corresponding  Orlicz spaces are not Lebesgue spaces. This example is taken from the \cite{rao} which was originally discovered by Riordan \cite{Rio}. For $1 < p < 2,$ we define function 
			$\Phi(x)= \frac{x^p}{p} \ln(x) \ln(\ln x)$ for $x \geq x_0$ and $\Phi(x)=  \frac{x^p}{p} \ln(\frac{1}{x}) \ln(\ln \frac{1}{x})$ for $x \leq x_1.$ We choose $x_0$ large enough and $x_1$ small enough so that $\Phi(x)$ gives a convex function by joining the points $(x_1, \Phi(x_1))$ and $(x_0, \Phi(x_0))$ by a straight line. If $q$ denotes the Lebesgue conjugate of $p,$ that is, $q= \frac{p}{p-1}$ then the function $\Psi$ is given by $\Psi(x)=\frac{x^q}{q} L(x)^{\frac{q}{p}},$ where  $L(x)= \ln(x) \ln(\ln x)$ for $x \geq x_0' $ and $\ln(\frac{1}{x}) \ln(\ln \frac{1}{x})$ for $x \leq x_1'.$ Then we again choose $x_0''$ large enough and $x_1'$ small enough so that $\Phi(x)$ gives a convex function by joining the points $(x_1, \Phi(x_1))$ and $(x_0, \Phi(x_0))$ by a straight line. Although,  $\Psi$ may not be the complementary function but it is equivalent to the complimentary function. 
			\item[(iii)] For $1 < p \leq 2, $ if $\Phi(x)= \frac{x^p}{p}$ and $\Psi(x)= \frac{x^q}{q}$ with $q= \frac{p}{p-1},$ then using the above method we can get usual Hausdorff-Young inequality for Lebesgue spaces on compact homogeneous manifolds. In fact, this method was used to solve maximal problem in the case for compact groups by Hirschman \cite{Hirschman} using the same norm on Lebesgue spaces as we considered in this paper . 
		\end{itemize}

	\end{rem}

	\section*{Acknowledgment}
	The authors are supported by Odysseus I Project (by FWO, Belgium). The second author is also supported by the Leverhulme Grant RPG-2017-151 and by EPSRC Grant EP/R003025/1.

\end{document}